\documentclass[graybox]{svmult}

\usepackage{helvet}         
\usepackage{courier}        
\usepackage{type1cm}        
\usepackage{makeidx}         
\usepackage{graphicx}        
\usepackage{multicol}        
\usepackage[bottom]{footmisc}
\usepackage{algorithm, algorithmic}
\usepackage{amsmath}
\usepackage{amssymb}\usepackage{mathrsfs}\usepackage{amscd}\usepackage{graphicx}\usepackage{subfigure}\usepackage{amsfonts}\usepackage{amsxtra}\usepackage{color}

\DeclareMathOperator{\sym}{sym}

\DeclareMathOperator*{\spann}{span}\DeclareMathOperator{\dist}{dist}
\DeclareMathOperator{\tr}{tr}

\DeclareMathOperator{\Pol}{Pol}

\newcommand{\G}{\mathcal{G}}

\newcommand{\R}{\mathbb{R}}\newcommand{\C}{\mathbb{C}}

\makeindex             

\begin{document}

\title*{Quasi Monte Carlo integration and kernel-based function approximation on Grassmannians}
\titlerunning{Diffusion polynomials on Grassmannians}
\author{Anna Breger, Martin Ehler, and Manuel Gr\"af}
\institute{Anna Breger \at University of Vienna, Department of Mathematics, Oskar-Morgenstern-Platz 1, A-1090 Vienna,  
\email{anna.breger@univie.ac.at}
\and Martin Ehler \at University of Vienna, Department of Mathematics, Oskar-Morgenstern-Platz 1, A-1090 Vienna,  \email{martin.ehler@univie.ac.at}
\and Manuel Gr\"af \at University of Vienna, Department of Mathematics, Oskar-Morgenstern-Platz 1, A-1090 Vienna,  \email{manuel.graef@univie.ac.at}}

%
%
\maketitle

\abstract{
Numerical integration and function approximation on compact Riemannian manifolds based on eigenfunctions of the Laplace-Beltrami operator have been widely studied in the recent literature. The standard example in numerical experiments is the Euclidean sphere. Here, we derive numerically feasible expressions for the approximation schemes on the Grassmannian manifold, and we present the associated numerical experiments on the Grassmannian. Indeed, our experiments illustrate and match the corresponding theoretical results in the literature.
}

\section{Introduction}
\label{sec:1}

The present paper is dedicated to numerical experiments concerning two classical problems in numerical analysis, numerical integration and function approximation. The novelty of our experiments is that we work on the Grassmannian manifold as an example of a compact Riemannian manifold illustrating theoretical results  in the recent literature. 

Indeed, recent data analysis methodologies involve kernel based approximation of
functions on manifolds and other measure spaces,
cf.~\cite{Maggioni:2008fk,Mhaskar:2010kx} and
\cite{Geller:2011fk,Geller:2012fk,Geller:2013uk}. The kernels are build up by
what is known as diffusion polynomials, which are eigenfunctions of elliptic
differential operators, commonly chosen as the Laplace-Beltrami operator when
dealing with compact Riemannian manifolds.

Numerical implementations of the approximation schemes require pointwise evaluation of the eigenfunctions.  
However, explicit formulas for eigenfunctions are only known in few special cases. 
If the manifold is the unit sphere $\mathbb{S}^{d-1}$, for instance, then the eigenfunctions of the spherical Laplacian are the spherical harmonics, which are indeed polynomials in the usual sense. The corresponding kernels for the sphere have been computed explicitly in \cite{Gia:2006fk,Gia:2008fk}. 

The kernel based approximation requires the computation of the corresponding integral operator, see \eqref{eq:approx sigma} in Section \ref{sec:2}. In the realm of numerical integration, the integral itself is usually approximated by a weighted sum over sample values, see also \cite{Filbir:2010aa,Filbir:2011fk}. The latter fits well to the common scenario when the target function needs to be approximated from a finite sample in the first place. 

Numerical integration on Euclidean spaces is a classical problem in numerical analysis. Recently, Quasi Monte Carlo (QMC) numerical integration on compact Riemannian manifolds has been studied in \cite{Brandolini:2014oz} from a theoretical point of view, see also \cite{Pesenson:2012fp}. If more and more samples are used, then the smoothness parameter of Bessel potential spaces steers the decay of the integration error. QMC integration has been introduced for the sphere in \cite{Brauchart:fk}, where many explicit examples are provided and extensive numerical experiments illustrate the theoretical claims. 

The major aim of the present paper is to provide numerical experiments for the above integration and approximation schemes when the manifold is the Grassmannian, i.e., the collection of $k$-dimensional subspaces in $\R^d$, naturally identified with the collection $\mathcal{G}_{k,d}$ of rank-$k$ orthogonal projectors on $\R^d$, cf.~\cite[Chapter 1]{Chikuse:2003aa}. 

Therefore, we require explicit formulas of the kernels used in \cite{Maggioni:2008fk,Mhaskar:2010kx}. Indeed, the degree of a diffusion polynomial, by definition, relates to the magnitude of the corresponding eigenvalues. We check that diffusion polynomials of degree at most $2t/\sqrt{k}$ are indeed usual multivariate polynomials of degree $t$ restricted to the Grassmannian. The explicit formula for the kernel is derived through generalized Jacobi polynomials. By computing cubature formulas on Grassmannians through some numerical minimization process, we are able to provide numerical experiments for the approximation of functions on Grassmannians and for the QMC integration on Grassmannians supporting the theoretical results in \cite{Brandolini:2014oz,Maggioni:2008fk,Mhaskar:2010kx}.


The outline is as follows: In Section \ref{sec:ints} we recall QMC integration from
\cite{Brandolini:2014oz,Brauchart:fk}, and we recall the approximation
scheme from \cite{Maggioni:2008fk,Mhaskar:2010kx} in the special case of the
Grassmannian in Section \ref{sec:2}. Section \ref{sec:4} provides feasible formulations for numerical experiments. Indeed, Section
\ref{sec:general G} is devoted to derive explicit formulas for the involved
kernel by means of generalized Jacobi polynomials. We check the relations
between diffusion polynomials and ordinary polynomials restricted to the
Grassmannian in Section \ref{sec:diff poly}, and we provide the framework for numerically computing cubatures in Grassmannians in Section \ref{sec:cubs}. The numerical experiments are
provided in Section \ref{sec:nums}.

\section{Quasi Monte Carlo integration}\label{sec:ints}
We identify the Grassmannian, the collection of $k$-dimensional subspaces in $\R^d$, with the set of orthogonal projectors on $\R^d$ of rank $k$, denoted by
\begin{equation*}
\G_{k,d} := \{ P \in \R^{d\times d}_{\sym} : P^{2}=P ;\; \tr(P)=k \}.
\end{equation*}
Here, $\R^{d\times d}_{\sym}$ is the set of symmetric matrices in
$\R^{d \times d}$ and $\tr(P)$ denotes the trace of $P$. The dimension of the Grassmannian is $\dim(\mathcal{G}_{k,d})=k(d-k)$. The canonical Riemannian measure on
$\G_{k,d}$ is denoted by $\mu_{k,d}$, which we assume to be normalized to one. Without loss of generality, we assume $k\leq \frac{d}{2}$ throughout since $\mathcal{G}_{d-k,d}$ can be identified with $\mathcal{G}_{k,d}$.

As a classical problem in numerical analysis, we aim to approximate the integral over a continuous function $f:\mathcal{G}_{k,d}\rightarrow\C$ by a finite sum over weighted samples, i.e., we consider points
$\{P_j\}_{j=1}^n\subset \mathcal{G}_{k,d}$ and nonnegative weights
$\{\omega_j\}_{j=1}^n$ such that
\begin{equation*}
\sum_{j=1}^n\omega_j f(P_j)\approx \int_{\mathcal{G}_{k,d}}f(P)\mathrm d\mu_{k,d}(P).
\end{equation*}
In order to quantify the error by means of the smoothness of $f$, we shall define Bessel potential spaces on $\mathcal{G}_{k,d}$, for which we need some preparation.
Let $\{\varphi_\ell\}_{\ell=0}^\infty$ be the collection of orthonormal
eigenfunctions of the Laplace-Beltrami operator $\Delta$ on $\mathcal{G}_{k,d}$,
and $\{-\lambda_\ell\}_{\ell=0}^\infty$ are the corresponding eigenvalues
arranged, so that $0=\lambda_0\leq \lambda_1\leq \ldots$. Without loss of generality, we choose each $\varphi_\ell$ to be real-valued, in particular, $\varphi_0\equiv 1$. 
The Fourier transform of $f\in L_p(\mathcal{G}_{k,d})$, where $1\leq p\leq\infty$, is defined by 
\begin{equation*}
\hat{f}(\ell):=\int_{\mathcal G_{k,d}}f(P) \varphi_{\ell}(P) \mathrm d \mu_{k,d}(P), \qquad \ell=0,1,2,\ldots.
\end{equation*}
Essentially following \cite{Brandolini:2014oz,Mhaskar:2010kx}, we formally define $(I-\Delta)^{s/2} f$ to be the distribution on $\mathcal{G}_{k,d}$, such that $\langle (I-\Delta)^{s/2} f,\varphi_\ell\rangle = (1+\lambda_\ell)^{s/2}\langle f,\varphi_\ell\rangle$, for all $\ell=0,1,2,\ldots$. The Bessel potential space $H^s_p(\mathcal{G}_{k,d})$, for $1\leq p\leq \infty$ and $s \ge 0$, is
\begin{align*}
H^s_p(\mathcal{G}_{k,d})&:=\{f\in L_p(\mathcal{G}_{k,d}) :  \| f\|_{H^s_p}<\infty\},\quad\text{where}\\
\| f\|_{H^s_p} &:= \| (I-\Delta)^{s/2} f  \|_{L_p},
\end{align*}
i.e., $f\in H^s_p(\mathcal{G}_{k,d})$ if and only if $f\in L_p(\mathcal{G}_{k,d})$ and $(I-\Delta)^s f\in L_p(\mathcal{G}_{k,d})$. Note that this definition is indeed consistent with \cite{Brandolini:2014oz,Mhaskar:2010kx}, see \cite[Theorem 2.1, Definition 2.2]{Brandolini:2014oz} in particular. For $s>k(d-k)/p$ with $1\leq p\leq \infty$, the space $H^s_p(\mathcal{G}_{k,d})$ is embedded into the space of continuous functions on $\mathcal{G}_{k,d}$, see, for instance,  \cite{Brandolini:2014oz}. For $1<p<\infty$, this embedding also follows from results on Bessel potential spaces on general Riemannian manifolds with bounded geometry, cf.~\cite[Theorem 7.4.5, Section 7.4.2]{Triebel:1992aa}, and on $\R^d$ with $1\leq p\leq \infty$, see \cite[Chapter V, 6.11]{Stein:1970yg}. According to \cite{Brandolini:2014oz}, for any sequence of points 
$\{P^t_j\}_{j=1}^{n_t}\subset\mathcal{G}_{k,d}$, $t=0,1,2,\ldots$, and positive weights
$\{\omega^t_j\}_{j=1}^{n_t}\subset\R$ with $n_t\rightarrow \infty$, there is a function $f\in H^s_p(\mathcal{G}_{k,d})$
\renewcommand*{\thefootnote}{\fnsymbol{footnote}}\setcounter{footnote}{1} such that \footnote{ We use the notation
  $\gtrsim$, meaning the right-hand side is less or equal to the left-hand side
  up to a positive constant factor. The symbol $\lesssim$ is used analogously,
  and $\asymp$ means both hold, $\lesssim$ and $\gtrsim$. If not explicitly
  stated, the dependence or independence of the constants shall be clear from
  the context.}
\begin{equation}\label{eq:negativ}
\Big |\int_{\mathcal{G}_{k,d}} f(P)\mathrm d\mu_{k,d}(P) -\sum_{j=1}^{n_t}\omega^t_j f(P^t_j) \Big| \gtrsim n_t^{-\frac{s}{k(d-k)}}\|f\|_{H^s_p},
\end{equation}
where the constant does not depend on $t$. Thus, we cannot do any better than
the rate
\begin{equation*}
n_t^{-\frac{s}{k(d-k)}}.
\end{equation*}

In order to quantify the quality of weighted point sequences
$\{(P^t_j,\omega^t_j)\}_{j=1}^{n_t}$, $t=0,1,2,\dots$, for numerical
integration, we make the following definition, whose analogous formulation on
the sphere (with constant weights) is due to \cite{Brauchart:fk}.
\begin{definition}
Given $s>k(d-k)/p$, a sequence $\{(P^t_j,\omega^t_j)\}_{j=1}^{n_t}$, $t=0,1,2,\ldots$,  of $n_t$ points in $\mathcal{G}_{k,d}$ and positive weights with $n_t\rightarrow\infty$ is called a \emph{sequence of Quasi Monte Carlo (QMC) systems for} $H^s_p(\mathcal{G}_{k,d})$ if
\begin{equation*}
\Big |\int_{\mathcal{G}_{k,d}} f(P)\mathrm d\mu_{k,d}(P)  -\sum_{j=1}^{n_t}\omega^t_j f(P^t_j)\Big| \lesssim n_t^{-\frac{s}{k(d-k)}}\|f\|_{H^s_p}
\end{equation*}
holds for all $f\in H^s_p(\mathcal{G}_{k,d})$.  
\end{definition}

In case $p=2$, given $s>k(d-k)/2$, any sequence of QMC systems  $\{(P^t_j,\omega^t_j)\}_{j=1}^{n_t}$ for $H^s_2(\mathcal{G}_{k,d})$ is also a sequence of QMC systems for $W^{s'}_2(\mathcal{G}_{k,d})$, for all $s'$ satisfying $s\geq s'>k(d-k)/2$, cf.~\cite{Brandolini:2014oz}. 

Especially for the integration of smooth functions, random points lack quality when compared to QMC systems. 
\begin{proposition}\label{th:random inde}
For $s>k(d-k)/2$, suppose $P_1,\ldots,P_n$ are random points on $\mathcal{G}_{k,d}$, independently identically distributed according to $\mu_{k,d}$ then it holds
\begin{equation*}
\sqrt{\mathbb{E} \Big[\sup_{\substack{f\in H^s_2(\mathcal{G}_{k,d})\\ \|f\|_{H^s_2}\leq 1}}  \Big|\int_{\mathcal{G}_{k,d}} f(P)\mathrm d\mu_{k,d}(P) - \frac{1}{n}\sum_{j=1}^n f(P_j)  \Big|^2\Big]}  = cn^{-\frac{1}{2}}
\end{equation*}
with $c^2=\sum_{\ell=1}^\infty (1+\lambda_\ell)^{-s}$.
\end{proposition}
Note that the condition $s>k(d-k)/2$ implies that
$\frac{s}{k(d-k)}>\frac{1}{2}$, so that on average QMC systems indeed perform
better than random points for smooth functions. The proof of Proposition \ref{th:random inde} is derived by following the lines
in \cite{Brauchart:fk}. In fact, the result is already contained in
\cite[Corollary 2.8]{Graf:2013zl}, see also \cite{Nowak:2010rr}, within a more
general setting.


In order to derive QMC systems, we shall have a closer look at cubature points, for which we need the space of
\emph{diffusion polynomials} of degree at most $t$, defined by
\begin{equation}\label{eq:poly diff}
\Pi_t:=\spann\{\varphi_\ell:\lambda_\ell\leq  t^2 \},
\end{equation}
see \cite{Mhaskar:2010kx} and references therein. 
\begin{definition}
  For $\{P_j\}_{j=1}^n\subset\mathcal{G}_{k,d}$ and positive weights
  $\{\omega_j\}_{j=1}^n$, we say that $\{(P_j,\omega_j)\}_{j=1}^n$ is a
  \emph{cubature for $\Pi_t$} if 
  \begin{equation}\label{eq:def cub 00}
    \int_{\mathcal{G}_{k,d}} f(P)\mathrm d\mu_{k,d}(P) = \sum_{j=1}^n \omega_j f(P_j),\quad \text{for all } f\in  \Pi_t.
  \end{equation}
  The number $t$ refers to the \emph{strength} of the cubature.
\end{definition}
In the following result, cf.~\cite[Theorem 2.12]{Brandolini:2014oz}, the cubature error is bounded by the  cubature strength $t$, not the number of points. 
\begin{theorem}\label{th:brandolini}
Suppose $s>k(d-k)/p$ and assume that $\{(P^t_j,\omega^t_j)\}_{j=1}^{n_t}$ is a cubature for $\Pi_t$. Then we have, for $f\in H^s_p(\mathcal{G}_{k,d})$,
\begin{equation*}
\Big|\int_{\mathcal{G}_{k,d}} f(P)\mathrm d\mu_{k,d}(P)-\sum_{j=1}^{n_t} \omega^t_j f(P^t_j)\Big| \lesssim  t^{-s}\|f\|_{H^s_p}.
\end{equation*}
\end{theorem}
Weyl's estimates on the spectrum of an elliptic operator yield 
\begin{equation*}
\dim(\Pi_t)\asymp t^{k(d-k)},
\end{equation*}
cf.~\cite[Theorem 17.5.3]{Hormander:1983gf}. This implies that any sequence of cubatures $\{(P^t_j,\omega^t_j)\}_{j=1}^{n_t}$ of strength $t$, respectively, must obey 
$
n_t\gtrsim t^{k(d-k)}
$ 
 asymptotically in $t$, cf.~\cite{Harpe:2005fk}. There are indeed sequences of cubatures $\{(P^t_j,\omega^t_j)\}_{j=1}^{n_t}$ of strength $t$, respectively, satisfying 
 \begin{equation}\label{eq:few cubs2}
n_t\asymp t^{k(d-k)},
\end{equation}
cf.~\cite{Harpe:2005fk}. In this case, Theorem \ref{th:brandolini} leads to
\begin{equation}\label{eq:general cub fomula}
\Big|\int_{\mathcal{G}_{k,d}} f(P)\mathrm d\mu_{k,d}(P)-\sum_{j=1}^{n_t} \omega^t_j f(P^t_j)\Big| \lesssim  n_t^{-\frac{s}{k(d-k)}}\|f\|_{H^s_p},
\end{equation}
so that we have settled that QMC systems do exist, for any $s>k(d-k)/p$, and can be derived via cubatures.

\begin{remark}
Cubature points $\{P_j\}_{j=1}^n$ for $\Pi_t$ with constant weights $\omega_j=\frac{1}{n}$ are called $t$-designs. 
  For all $t=1,2,\ldots$, there exist $t$-designs, cf.~\cite{Seymour:1984bh}.
The results in \cite{Bondarenko:2011kx} imply that there are $t$-designs satisfying \eqref{eq:few cubs2} provided that $k=1$. However, for $2\leq k\leq \frac{d}{2}$, it is still an open problem if the asymptotics
  \eqref{eq:few cubs2} can be achieved by $t$-designs in place of cubatures.
\end{remark}

\section{Approximation by diffusion kernels}
\label{sec:2}

One example, where integrals over the Grassmannian are replaced with weighted finite sums, is the approximation of a function $f:\mathcal{G}_{k,d}\rightarrow\C$ from finitely many samples. The approximation scheme developed in \cite{Maggioni:2008fk,Mhaskar:2010kx} works for manifolds and metric measure spaces in general, but we shall restrict the presentation to the Grassmannian. 

For a function $f \in L_{p}(\mathcal G_{k,d})$, we denote the
(polynomial) \emph{best approximation error} by
\begin{equation*}
  \dist(f,\Pi_{t})_{L_p}:=\inf_{g\in\Pi_{t}} \|f-g\|_{L_p},
\end{equation*}
where $t\geq 0$. 
It is possible to quantify the best approximation error in dependence of the
function's smoothness, see \cite[Proposition 5.3]{Mhaskar:2010kx} for the following result. 
\begin{theorem}\label{th:00}
If $f\in H^s_p(\mathcal{G}_{k,d})$, then 
\begin{equation*}
\dist(f,\Pi_t)_{L_p}\lesssim  t^{-s} \|f\|_{H^s_p}.
\end{equation*}
\end{theorem}

Given $f\in H^s_p(\mathcal{G}_{k,d})$ we now construct a particular
sequence of functions $\sigma_t(f)\in\Pi_t$, $t=1,2,\ldots$, that realizes this
best approximation rate. Note that, since the collection $\{\varphi_\ell\}_{\ell=0}^\infty$ is an orthonormal
basis for $L_2(\mathcal{G}_{k,d})$, any function $f\in L_2(\mathcal{G}_{k,d})$ can be expanded as a Fourier series
by 
\begin{equation*}
f =\sum_{\ell=0}^\infty \hat{f}(\ell) \varphi_\ell.
\end{equation*}
The approach in
\cite{Maggioni:2008fk,Mhaskar:2010kx} makes use of a smoothly truncated Fourier
expansion of $f$, 
\begin{equation*}
  \sigma_t(f) := \sum_{\ell=0}^\infty h(t^{-2}\lambda_\ell) \hat{f}(\ell) \varphi_\ell\in\Pi_t,
\end{equation*}
where $h:\R_{\geq 0}\rightarrow\R$ is an infinitely often differentiable and
nonincreasing function with $h(x)=1$, for $x\le 1/2$, and $h(x)=0$, for
$x\ge 1$. Using the kernel $K_t$ on $\G_{k,d} \times \G_{k,d}$ defined by
\begin{equation}\label{eq:sigma again and new}
  K_t(P,Q)= \sum_{\ell=0}^\infty h(t^{-2}\lambda_\ell) \varphi_\ell(P)\varphi_\ell(Q)
\end{equation}
we arrive after interchanging summation and integration at the following
alternative representation
\begin{equation}\label{eq:approx sigma}
\sigma_{t}(f) = \int_{\mathcal{G}_{k,d}} f(P) K_t(P,\cdot) \mathrm d\mu_{k,d}(P).
\end{equation}
Note, that the function $\sigma_{t}(f)$ is well-defined for general
$f\in L_p(\mathcal{G}_{k,d})$, $1\le p \le \infty$, and it turns out that
$\sigma_t(f)$ approximates $f$ up to a constant as good as the best approximation from $\Pi_t$,
cf.~\cite[Proposition 5.3]{Mhaskar:2010kx}.
\begin{theorem}\label{th:0}
If $f\in H^s_p(\mathcal{G}_{k,d})$, then 
\begin{equation*}
 \|f - \sigma_t(f)\|_{L_p} \lesssim t^{-s} \|f\|_{H^s_p}.
\end{equation*}
\end{theorem}

If $f$ needs to be approximated from a finite sample, then $\sigma_t(f)$ in \eqref{eq:approx sigma} cannot be determined  directly and is replaced with a weighted finite sum in \cite{Mhaskar:2010kx}. Indeed, for sample points $\{P_j\}_{j=1}^{n}\subset\mathcal{G}_{k,d}$ and weights $\{\omega_j\}_{j=1}^{n}$, we define
\begin{equation}\label{eq:def sigma tilde etc}
\sigma_t(f,\{(P_j,\omega_j)\}_{j=1}^n):=\sum_{j=1}^n \omega_j f(P_j)K_t(P_j,\cdot).
\end{equation}
Note that we must now consider functions $f$ in Bessel potential spaces, for which point evaluation makes sense. We shall observe in the following that if samples and weights satisfy some cubature type property, then the approximation rate is still preserved when using $\sigma_t(f,\{(P_j,\omega_j)\}_{j=1}^n)$ in place of $\sigma_{t}(f)$. However, we need an additional technical assumption on the points $\{P_j\}_{j=1}^{n}$, for which we denote the geodesic distance between
$P,Q\in\mathcal{G}_{k,d}$ by
\begin{equation*}
\rho(P,Q)=\sqrt{\theta_1^2+\ldots+\theta_k^2},
\end{equation*}
where $\theta_1,\ldots\theta_k$
are the principal angles between the associated subspaces of $P$ and $Q$
respectively, i.e., 
\begin{equation*}
\theta_i=\arccos(\sqrt{y_i}), \quad i=1,\ldots,k,
\end{equation*}
and $y_1,\ldots,y_k$ are the $k$ largest eigenvalues of the matrix $PQ$. We define the ball of radius $r$ centered around $P\in \mathcal{G}_{k,d}$ by 
\begin{equation*}
\mathbb{B}_r(P):=\{Q \in \mathcal G_{k,d} : \rho(P,Q) \le r \}. 
\end{equation*}

The following approximation from finitely many sample points is due to \cite[Proposition 5.3]{Mhaskar:2010kx}. \begin{theorem}\label{th:m 0}
For $t=1,2,\ldots$, suppose we are given a sequence of point sets $\{P^t_j\}_{j=1}^{n_t}\subset\mathcal{G}_{k,d}$ and positive weights $\{\omega_j^t\}_{j=1}^{n_t}$ such that 
\begin{equation}\label{eq:product cub}
\int_{\mathcal{G}_{k,d}} g_1(P)g_2(P)\mathrm d\mu_{k,d}(P) =\sum_{j=1}^{n_t} \omega_j^t g_1(P^t_j)g_2(P^t_j) ,\qquad g_1,g_2\in\Pi_t.
\end{equation}
Then the approximation error, for $f\in H^s_\infty(\mathcal{G}_{k,d})$, is bounded by
\begin{equation}\label{eq:error approx finale}
\|f-\sigma_t(f,\{(P^t_j,\omega^t_j)\}_{j=1}^{n_t})\|_{L_\infty} \lesssim t^{-s} (\|f\|_{L_\infty}+ \|f\|_{H^s_\infty}).
\end{equation}
\end{theorem}
Note that the original result stated in \cite{Mhaskar:2010kx} requires an additional regularity condition on the samples. This condition is satisfied since we restrict us to positive weights, cf.~\cite[Theorem 5.5 (a)]{Filbir:2011fk}. The assumption \eqref{eq:product cub} is a cubature type condition, for which our results in Section \ref{sec:diff poly} shall provide further clarification. It indeed turns out that there are sequences $\{(P^t_j,\omega^t_j)\}_{j=1}^{n_t}$ satisfying \eqref{eq:product cub} with $n_t\asymp t^{k(d-k)}$, in which case \eqref{eq:error approx finale} becomes
\begin{equation}\label{eq:rate in n finale}
\|f-\sigma_t(f,\{(P^t_j,\omega^t_j)\}_{j=1}^{n_t})\|_{L_\infty} \lesssim n_t^{-\frac{s}{k(d-k)}} (\|f\|_{L_\infty}+ \|f\|_{H^s(L_\infty)}).
\end{equation}
Note that the approximation rate in \eqref{eq:rate in n finale} matches the one in \eqref{eq:general cub fomula} for the integration error. The proof of Theorem \ref{th:m 0} in \cite{Mhaskar:2010kx} is indeed based on Theorem \ref{th:0} and on the approximation of the integral $\sigma_t(f)$ in \eqref{eq:approx sigma} by the weighted finite sum $\sigma_t(f,\{(P^t_j,\omega_j^t)\}_{j=1}^n)$ in \eqref{eq:def sigma tilde etc}. For related results on local smoothness and approximation, we refer to \cite{Ehler:2012fk}. 

\section{Numerically feasible formulations}\label{sec:4}
This section is dedicated to turn the approximation schemes presented in the
previous sections into numerically feasible expressions. In other words, we
determine explicit expressions for the kernel $K_t$ in \eqref{eq:sigma again and
  new} and provide an optimization method for the numerical
  computation of cubature points or QMC systems on the Grassmannian.

\subsection{Diffusion kernels on Grassmannians}\label{sec:general G}

The probability measure $\mu_{k,d}$ is invariant under orthogonal conjugation and induced by the Haar
(probability) measure $\mu_{\mathcal O(d)}$ on the orthogonal group
$\mathcal{O}(d)$, i.e., for any $Q \in \G_{k,d}$ and measurable function
$f$, we have
\[
\int_{\G_{k,d}} f(P)\mathrm d\mu_{k,d}(P) = \int_{\mathcal O(d)} f(O Q O^{\top}) \mathrm d
\mu_{\mathcal O(d)}(O).
\]
By the orthogonal invariance of the Laplace-Beltrami operator $\Delta$ on
$\G_{k,d}$ it is convenient for the description of the eigenfunctions to recall
the irreducible decomposition of $L_{2}(\G_{k,d})$ with respect to the
orthogonal group. Given a nonnegative integer $t$, a partition of $t$ is an
integer vector $\pi=(\pi_1,\dots,\pi_t)$ with $\pi_1\geq \ldots\geq \pi_t\geq 0$
and $|\pi| = t$, where $|\pi|:=\sum_{i=1}^t\pi_i$ is the size of $\pi$. The
length $l(\pi)$ is the number of nonzero parts of $\pi$.
The space $L_{2}(\G_{k,d})$ decomposes into 
\begin{equation}\label{eq:decomp L2 single}
L_{2}(\G_{k,d}) = \bigoplus_{l(\pi) \le k} H_{\pi}(\G_{k,d}),\qquad
H_{\pi}(\G_{k,d}) \perp H_{\pi'}(\G_{k,d}), \quad \pi \ne \pi',
\end{equation}
where $H_{\pi}(\G_{k,d})$ is equivalent to $\mathcal H_{2\pi}^{d}$, the
irreducible representation of $\mathcal O(d)$ associated to the partition
$2\pi=(2\pi_{1},\dots,2\pi_{t})$,
cf.~\cite{Bachoc:2002aa,James:1974aa}. 

By orthogonal invariance the spaces $H_{\pi}(\G_{k,d})$ are eigenspaces of the
Laplace-Beltrami operator $\Delta$ on $\mathcal{G}_{k,d}$ where, according to
\cite[Theorem 13.2]{James:1974aa}, the associated eigenvalues are
\begin{equation}\label{eq:eig values explicit} 
\lambda(\pi) = 2|\pi|d+4\sum_{i=1}^k \pi_i(\pi_i-i).
\end{equation}
Note, for a given eigenvalue $\lambda_{\ell}$ the corresponding eigenspace can
decompose into more than one irreducible subspace $H_{\pi}(\G_{k,d})$.

Note that the following results are translations from representation theory used in \cite{James:1974aa}, see also \cite{Bachoc:2002aa,Bachoc:2006aa}, into the terminology of reproducing kernels, where we have only adapted the scaling of the kernels.  
The space $H_{\pi}(\G_{k,d})$ equipped with the $L_{2}$ inner product is a finite
dimensional reproducing kernel Hilbert space, and its reproducing kernel $K_\pi$ is
given by
\begin{equation}\label{eq:rep of H etc}
K_\pi(P,Q)=\sum_{\varphi_\ell\in H_{\pi}(\G_{k,d})} \varphi_\ell(P)\varphi_\ell(Q).
\end{equation}
Moreover, $K_\pi$ is zonal, i.e., the value $K_\pi(P,Q)$
only depends on the $k$ largest eigenvalues
\begin{equation*}
y_1(P,Q),\ldots,y_k(P,Q),
\end{equation*}
of the matrix $PQ$ counted with multiplicities, see \cite{James:1974aa}. It
follows that the kernel $K_t$ in \eqref{eq:sigma again and new} is also zonal
since it can be written as
\begin{equation}\label{eq:kernel by K_pi}
  K_t(P,Q)= \sum_{l(\pi)\leq k} h(t^{-2}\lambda(\pi)) K_\pi(P,Q).
\end{equation}
%
%
%
%
According to \cite{James:1974aa}, the kernels $K_\pi$ are in one-to-one
correspondence with generalized Jacobi polynomials. For parameters
$\alpha,\beta \in \R$ satisfying
$\tfrac 1 2(m-1) < \alpha < \beta - \tfrac 1 2(m-1)$, the generalized Jacobi
polynomials, $J^{\alpha,\beta}_{\pi}:[0,1]^{m} \to \R$ with $l(\pi) \le m$, are
symmetric polynomials of degree $|\pi|$ and form a complete orthogonal system
with respect to the density
\begin{equation}
  \label{eq:w_ac}
  w_{\alpha,\beta}(y_{1},\dots,y_{m}) := \prod_{i=1}^{m}\big(y_{i}^{\alpha - \frac 1
    2(m+1)}(1-y_{i})^{\beta-\alpha-\frac 1 2(m+1)}\big) \prod_{j=i+1}^{m}|y_{i}-y_{j}|,
\end{equation}
where $0 < y_1,\ldots y_{m} < 1$, cf.~\cite{Davis:1999zg,Dumitriu:2007ax}. For
the special parameters $\alpha=\tfrac k 2$ with $k\leq \tfrac d 2$ and
$\beta=\tfrac d 2$, and the normalization
$J_{\pi}^{\frac k2,\frac d2}(1,\dots,1) = \dim(H_{\pi}(\G_{k,d}))$ the
generalized Jacobi polynomials in $m=k$ variables can be identified with the
reproducing kernels $K_\pi$ of $H_{\pi}(\G_{k,d}) $, i.e.,
 \begin{equation}
    \label{eq:pkl_jacobi}
K_\pi(P,Q) = J_{\pi}^{\frac k2, \frac d 2}( y_{1}(PQ), \dots, y_{k}(PQ) ), \qquad P,Q \in \G_{k,d}.
  \end{equation}
Now, \eqref{eq:kernel by K_pi} and \eqref{eq:pkl_jacobi} yield that the expression for the kernel $K_t$ in \eqref{eq:sigma again and new} can be computed explicitly by 
\begin{align*}
  K_t(P,Q)&=  \sum_{l(\pi)\leq k} h(t^{-2}\lambda(\pi))  J_{\pi}^{\frac k2, \frac d 2}( y_{1}(PQ), \dots, y_{k}(PQ) ).
\end{align*}
Thus, avoiding the actual computation of $\{\varphi_\ell\}_{\ell=0}^\infty$, we have derived the expression of $K_t$ by means of generalized Jacobi polynomials. 


\subsection{Diffusion polynomials on Grassmannians}\label{sec:diff poly}
This section is dedicated to investigate on the relations between diffusion polynomials $\Pi_t$ and multivariate polynomials of degree $t$ restricted to the Grassmannian. Indeed, the space of polynomials on $\G_{k,d}$ of degree at most $t$ is defined as restrictions of polynomials by 
\begin{equation}\label{eq:Pol def}
  \Pol_{t}(\G_{k,d}):=   \{ f|_{\G_{k,d}}: f\in\C[X]_t\},
\end{equation}
where $\C[X]_t$ is the collection of multivariate polynomials of degree at most
$t$ with $d^{2}$ many variables arranged as a matrix $X\in\C^{d\times d}$. Here, $f|_{\G_{k,d}}$ denotes the restriction of $f$ to $\G_{k,d}$. 
It turns out that $\Pol_{t}(\G_{k,d})$ is a direct sum of eigenspaces of the Laplace-Beltrami operator, i.e., 
\begin{equation}\label{eq:pol is sum}
 \Pol_{t}(\G_{k,d})= \bigoplus_{\substack{|\pi|\leq t \\ l(\pi)\leq k}} H_{\pi}(\G_{k,d}),
\end{equation}
cf.~\cite[Corollary in Section 11]{James:1974aa} and also \cite[Section 2]{Bachoc:2004fk}, 
which enables us to relate diffusion polynomials to regular polynomials restricted to the Grassmannian. For $k=1$, the eigenvalues \eqref{eq:eig values explicit} directly lead to 
\begin{equation*}
\Pi_{\sqrt{4t^2+2t(d-2)}}=\Pol_t(\mathcal{G}_{1,d}).
\end{equation*}
For general $k$, the situation is more complicated and needs some preparation.
\begin{lemma} \label{lemma:11}
  Let $d, k, t \in \mathbb N$ with $k \le \tfrac d2$ be fixed. Then for any
  partition $\pi$ with $l(\pi) \le k$ and $|\pi| \ge t$ it holds
  \begin{equation}
    \label{eq:lowlambda}
    \lambda(\pi) \ge 
    \lceil\tfrac 4k t^{2} + 2t(d - k -1)\rceil.
  \end{equation}
\end{lemma}
Note that the right-hand side of \eqref{eq:lowlambda}, up to the square root and the ceiling function, is \eqref{eq:eig values explicit} with $\pi_i=t/k$, for $i=1,\ldots,k$. 
\begin{proof}
In view of \eqref{eq:eig values explicit}, let us define
\begin{equation*} f(x_{1},\dots,x_{k})  := 2 d \sum_{i=1}^{k} x_{i} +
      4 \sum_{i=1}^{k} x_{i}(x_{i}-i).
      \end{equation*}
 For partitions $\pi$ with $|\pi| \ge t$, we obtain the lower bound
  \eqref{eq:lowlambda} by solving the following convex optimization problem
  \[
    \begin{aligned}
      \min_{x \in \mathbb R^{k}} f(x_{1},\dots,x_{k}) \quad     \text{ such that } \quad
      g_{i}(x)  \le 0,\;\; i=0,\dots,k,
    \end{aligned}  
  \]
  where $ g_{0}(x)  = t - \sum_{i=1}^{k} x_{i}$ and 
\begin{equation*}  
      g_{i}(x)  = x_{i+1}-x_{i}, \quad i=1,\dots,k-1,\qquad 
      g_{k}(x)  = -x_{k}.
\end{equation*}  Indeed, we shall verify that the
  minimum is attained at $x^{*}:=(\tfrac tk,\dots,\tfrac tk)$ with 
  \[
    f(x^{*}) = \tfrac 4k t^{2} + 2(d - k -1)t
  \] 
by checking the Karush-Kuhn-Tucker (KKT) conditions
  \[
    \begin{aligned}
      & \nabla f(x^{*}) + \sum_{i=0}^{k} \mu_{i} \nabla g_{i}(x^{*}) = 0,\\
      & g_{i}(x^{*}) \le 0, \quad \mu_{i} \ge 0, \quad \mu_{i}g_{i}(x^{*}) = 0, \qquad i = 0,\dots,k,
    \end{aligned}
  \]
  with $ \mu_{0} = 8\tfrac tk + 2d - 2(k+1)$ and $\mu_{i} = 2 i(k-i)$, for $i=1,\dots,k$. More precisely, denoting the canonical basis in $\mathbb R^{k}$ by $\{e_{i}\}_{i=1}^k$, we obtain
    {  \allowdisplaybreaks
\begin{align*}
      -\sum_{i=0}^{k} \mu_{i} \nabla g_{i}(x^{*})
      & = \mu_{0} \sum_{i=1}^{k} e_{i} - \sum_{i=1}^{k-1} \mu_{i} (e_{i+1} - e_{i}) \\
      & = (\mu_{0} + \mu_{1})e_{1} + (\mu_{0} - \mu_{k-1}) e_{k} + \sum_{i=2}^{k-1} (\mu_{0} + \mu_{i} - \mu_{i-1}) e_{i} \\
      & = (\mu_{0} + 2(k-1)) e_{1} + (\mu_{0} - 2(k-1)) e_{k} + \sum_{i=2}^{k-1} (\mu_{0} - 4(i-1) + 2(k-1)) e_{i}\\
      & = \sum_{i=1}^{k} (\mu_{0} - 4 i + 2(k+1)) e_{i}\\
      & = \sum_{i=1}^{k} (8 \tfrac tk + 2d - 4 i) e_{i}= \nabla f(x^{*})
\end{align*}
}
and conclude that the KKT-conditions are satisfied. Hence, 
  \eqref{eq:lowlambda} holds.
\end{proof}

\begin{theorem}\label{pro:1}
Polynomials and diffusion polynomials on the Grassmannian $\G_{k,d}$ satisfy the relation
\begin{equation*}
\Pi_{s(t+1)-\epsilon} \subset  \Pol_{t}(\G_{k,d}) \subset \Pi_{\sqrt{4t^2+2t(d-2)}}, \quad \text{for all $0<\epsilon<2s(t+1)$},
\end{equation*}
where $s(t)=\sqrt{\lceil\tfrac 4k t^{2} + 2t(d - k -1)\rceil}$.
\end{theorem}
\begin{proof}
Due to \eqref{eq:decomp L2 single}, we are only dealing with partitions $\pi$ satisfying $l(\pi)\leq k$. For $|\pi|\leq t$, we derive 
\begin{align*}
\lambda(\pi)=2|\pi|d+4\sum_{i=1}^k \pi_i(\pi_i-i) & \leq 2|\pi|d+4\sum_{i=1}^k \pi^2_i-4\sum_{i=1}^k \pi_i \\
&\leq 4t^2+2t(d-2),
\end{align*}
which yields the second set inclusion. 

Lemma \ref{lemma:11} yields that $\lambda(\pi)<s^2(t+1)$ implies $|\pi|< t+1$, the latter being equivalent to $|\pi|\leq t$ since both $|\pi|$ and $t$ are integers. The range of $\epsilon$ yields $(s(t+1)-\epsilon)^2<s^2(t+1)$, so that we deduce the first set inclusion.
\end{proof}
Asymptotically in $t$, diffusion polynomials of order $\frac{2}{\sqrt{k}}t$ are indeed polynomials of degree at most $t$, and Theorem \ref{th:0} yields, for $f\in H^s_p(\mathcal{G}_{k,d})$,
\begin{equation*}
\dist(f,\Pol_t(\mathcal{G}_{k,d}))_{L_p}\asymp\dist(f,\Pi_t)_{L_p} \lesssim t^{-s} \|f\|_{H^s_p}.
\end{equation*}
For related further studies on $\dist(f,\Pol_t(\mathcal{G}_{k,d}))$, see \cite{Ragozin:1970zr}.

In view of Theorem \ref{pro:1}, we shall also define cubatures for $\Pol_t(\mathcal{G}_{k,d})$.
\begin{definition}\label{def:def 0}
  For $\{P_j\}_{j=1}^n\subset\mathcal{G}_{k,d}$ and positive weights
  $\{\omega_j\}_{j=1}^n$, we say that $\{(P_j,\omega_j)\}_{j=1}^n$ is a
  \emph{cubature for $\Pol_t(\mathcal{G}_{k,d})$} if 
  \begin{equation}\label{eq:def cub 001}
    \int_{\mathcal{G}_{k,d}} f(P)\mathrm d\mu_{k,d}(P) = \sum_{j=1}^n \omega_j f(P_j),\quad \text{for all } f\in  \Pol_t(\mathcal{G}_{k,d}).
  \end{equation}
  We say that the points $\{P_j\}_{j=1}^n\subset\mathcal{G}_{k,d}$ are a $t$-design for $\Pol_t(\mathcal{G}_{k,d})$ if \eqref{eq:def cub 001} holds for constant weights $\omega_1=\ldots=\omega_n=1/n$.
  \end{definition}
It turns out that the numerical construction of cubature points and $t$-designs for $\Pol_t(\mathcal{G}_{k,d})$ is somewhat easier than for $\Pi_t$ directly, which is outlined in the subsequent section.

\begin{remark}\label{rem:001}
Since $\Pol_t(\mathcal{G}_{k,d})$ are restrictions of ordinary polynomials, we observe $\Pol_{t_1}(\mathcal{G}_{k,d})\cdot \Pol_{t_2}(\mathcal{G}_{k,d})\subset \Pol_{t_1+t_2}(\mathcal{G}_{k,d})$. Thus, $\{(P^t_j,\omega_j^t)\}_{j=1}^{n_t}$ being cubatures for $\Pol_{2t}(\mathcal{G}_{k,d})$ yields that \eqref{eq:product cub} is satisfied when $\Pi_t$ is replaced with $\Pi_{s(t+1)-\epsilon}$. The latter implies that we must then also replace $\sigma_t(f,\{(P^t_j,\omega^t_j)\}_{j=1}^{n_t})$ with $\sigma_{s(t+1)-\epsilon}(f,\{(P^t_j,\omega^t_j)\}_{j=1}^{n_t})$ in Theorem \ref{th:m 0}. 
\end{remark}

For general $k$, the second set inclusion in Theorem \ref{pro:1} is sharp because $\lambda(t,0,\ldots,0)=4t^2+2t(d-2)$. The first set inclusion in Theorem \ref{pro:1} may only be optimal for $t$ being a multiple of $k$. To prepare for our numerical experiments later, we shall investigate on $\mathcal{G}_{2,d}$ more closely.

\begin{theorem}\label{th:new and latest}
For $k=2$, we obtain
\begin{equation*}
\Pi_{s(t+1)-\epsilon} \subset \Pol_t(\mathcal{G}_{2,d}), \quad \text{for all } 0<\epsilon<2s(t+1),
\end{equation*}
where $s(t)= \sqrt{2t^{2} + 2t(d -3) + 2(1 + (-1)^{t+1})}$.
\end{theorem}
Note that $s(t)$ in Theorem \ref{th:new and latest} satisfies $s^2(t)=\lambda(\lceil \tfrac t2 \rceil, \lfloor \tfrac t2
    \rfloor )$. 
It matches the definition in Theorem \ref{pro:1} provided that $t$ is even. For odd $t$, $s(t)$ in Theorem \ref{th:new and latest} is indeed larger than in Theorem \ref{pro:1}, and the difference of the squares is $4$.
\begin{proof}
Any
  partition $\pi$ of length $k=2$ with $|\pi| = t$ can be parameterized by
  $\pi(r) = (t - r, r)$, $r = 0,\dots, \lfloor \tfrac t2 \rfloor$. We have checked that $\lambda(\pi(r))$ is a quadratic function in $r$, which is strictly decreasing in $r$. Observing
  furthermore that $\lambda(\pi') \ge \lambda(\pi)$ if $\pi_{i}' \ge \pi_{i}$,
  $i=1,\dots,k$, we infer that $|\pi|\geq t$ implies
  \begin{equation}
    \label{eq:lowlambda2}
    \lambda(\pi) \ge \lambda(\lceil \tfrac t2 \rceil, \lfloor \tfrac t2
    \rfloor ) = s^2(t).
  \end{equation}
  Therefore, $\lambda(\pi)<s^2(t+1)$ implies $|\pi|\leq t$ since $|\pi|$ and $t$ are integers.
\end{proof}

\begin{example}\label{ex:first one}
The particular case $\mathcal G_{2,4}$ yields 
\begin{equation*}
\Pi_{s(t+1)-\epsilon}\subset \Pol_t(\mathcal{G}_{2,4}), \quad\text{ for all $0<\epsilon<2s(t+1)$,}
\end{equation*}
where $s(t):=\sqrt{2t^{2} + 2t + 2(1 + (-1)^{t+1})}$, which implies 
\begin{equation*}
\Pi_{s(t+1)-\epsilon} \cdot \Pi_{s(t+1)-\epsilon} \subset  \Pol_{2t}(\mathcal{G}_{2,4}).
\end{equation*}
Thus, given a cubature for $\Pol_{2t}(\mathcal{G}_{2,4})$, the condition \eqref{eq:product cub} in Theorem \ref{th:m 0} is satisfied with respect to $\Pi_{s(t+1)-\epsilon}$. 
\end{example}

\subsection{Worst case error of integration on Grassmannians}\label{sec:cubs}

Given some subspace $\mathcal{H}$ of continuous functions on $\mathcal{G}_{k,d}$ the \emph{worst case error} of integration (with
respect to some norm $\|\cdot\|$ on $\mathcal{H}$) for points $\{P_{j}\}_{j=1}^{n}\subset\mathcal{G}_{k,d}$ and weights $\{\omega_{j}\}_{j=1}^{n}$ is defined by 
\begin{equation*}
  \mathrm{wce}_{\mathcal{H},\|\cdot\|}(\{(P_{j},\omega_{j})\}_{j=1}^{n}) := \sup_{\substack{f \in \mathcal{H}\\
      \|f\|=1} } \Big| \int_{\G_{k,d}}f(P) \mathrm d\mu_{k,d}(P) -
  \sum_{j=1}^{n} \omega_{j} f(P_{j}) \Big|,
\end{equation*}
see also \cite{Graf:2013zl,Nowak:2010rr}. 
If $\mathcal{H}=H_K$ is a reproducing kernel Hilbert space, whose reproducing kernel $K$ is
\begin{equation*}
K(P,Q) = \sum_{ l(\pi)\leq k} r_{\pi} K_{\pi}(P,Q)
  = \sum_{ l(\pi)\leq k} r_{\pi} \sum_{\lambda_{\ell} =
    \lambda(\pi)}\varphi_{\ell}(P) \varphi_{\ell}(Q) , \qquad P,Q \in \G_{k,d},
\end{equation*}
with $r_{\pi} \geq  0$, $|\pi| \ge 0$, and sufficient decay of the coefficients, then the associated inner product is
\begin{equation*}
  (f,g)_{K} = \sum_{\substack{l(\pi)\leq k\\ r_\pi>0}} r_{\pi}^{-1} \sum_{\lambda_{\ell} = \lambda(\pi)}\hat f(\ell) \overline{\hat g(\ell)},
\end{equation*}
and the Riesz representation theorem yields
\begin{align}
\mathrm{wce}_{H_K,\|\cdot\|_{K}}(\{(P_{j}, \omega_{j})\}_{j=1}^{n})^{2} 
    & = \sum_{l(\pi)\leq k} r_{\pi} \sum_{\lambda_{\ell} =
      \lambda(\pi)} \Big|\int_{\G_{k,d}}\varphi_{\ell}(P) \mathrm d \mu_{k,d}(P) - \sum_{j=1}^{n}w_{j}\varphi_{\ell}(P_{j}) \Big|^{2} \nonumber \\
      & = r_{(0)} - 2 r_{(0)} \sum_{j=1}^{n} \omega_{j} +
    \sum_{i,j=1}^{n}\omega_{i}\omega_{j} K(P_{i},P_{j}).\label{eq:wce_formula}
\end{align}
Note that the worst case error is a weighted $\ell_{2}$-average of the
integration errors of the basis functions $\varphi_{0},\varphi_1,\varphi_2,\ldots$.

Recall, for instance, $H^s_2(\mathcal{G}_{k,d})$ is a Hilbert space with inner product
\begin{equation}\label{eq:inner product on W}
\langle f,g\rangle_{H^s_2} = \sum_{\ell=0}^\infty (1+\lambda_\ell)^{s} \hat{f}(\ell)\overline{\hat{g}(\ell)},\quad f,g\in H^s_2(\mathcal{G}_{k,d}),
\end{equation}
and the Bessel kernel on the Grassmannian is $K^s_B:\mathcal{G}_{k,d}\times \mathcal{G}_{k,d} \rightarrow \R$ with
\begin{align*}
K^s_B(P,Q)&=\sum_{\ell=0}^\infty (1+\lambda_\ell)^{-s}\varphi_\ell(P)\varphi_\ell(Q)\\
&=\sum_{l(\pi)\leq k}^\infty (1+\lambda(\pi))^{-s}J^{\frac k2, \frac d 2}_\pi(y_1(P,Q),\ldots,y_k(P,Q)).
\end{align*}
If $s>k(d-k)/2$, then it is easily checked that $K^s_B$ is the reproducing kernel for $H^s_2(\mathcal{G}_{k,d})$ with respect to the inner product  \eqref{eq:inner product on W}, see also \cite{Brandolini:2014oz}.

Note that the polynomial space $\Pol_{t}(\G_{k,d})$ is also a reproducing kernel Hilbert space. Indeed, 
given a partition $\pi$ with $|\pi|\leq t$ and $l(\pi)\leq k$, the reproducing kernel of $H_\pi(\mathcal{G}_{k,d})$ with respect to the $L_2$ inner product is $K_\pi$ in \eqref{eq:rep of H etc}. Due to \eqref{eq:pol is sum}, the reproducing kernels for $\Pol_{t}(\G_{k,d})$ are exactly 
\[
  R_{t}(P,Q) = \sum_{\substack{|\pi|\le t \\ l(\pi)\leq k}} r_{\pi} K_{\pi}(P,Q)
  = \sum_{\substack{|\pi|\le t \\ l(\pi)\leq k}} r_{\pi} \sum_{\lambda_{\ell} =
    \lambda(\pi)}\varphi_{\ell}(P) \varphi_{\ell}(Q) , \qquad P,Q \in \G_{k,d},
\]
with $r_{\pi} > 0$, $|\pi| \ge 0$. Note that $R_t$ is indeed reproducing as a finite linear combination with nonnegative coefficients of reproducing kernels, and it reproduces  $\Pol_{t}(\G_{k,d})$ because of \eqref{eq:pol is sum} and none of the coefficients vanish. 
Now, by Definition \ref{def:def 0} any cubature for $\Pol_{t}(\G_{k,d})$ has
zero worst case error independent of the chosen norm, and thus
independent of $R_{t}$. A particularly simple reproducing kernel for $\Pol_{t}(\G_{k,d})$ is
\[
  R_{t}(P,Q) = \tr(PQ)^{t}, \qquad P,Q \in \G_{k,d},
\] 
see, for instance, \cite{Ehler:2014zl}.  
Hence, formula \eqref{eq:wce_formula} provides us with a simple method to
numerically compute cubature points by some minimization method. In particular,
$t$-designs $\{ (P_{j},\frac1n)\}_{j=1}^{n}$ are constructed by minimizing 
\[
  \frac{1}{n^{2}} \sum_{i,j=1}^{n} \tr(P_{i},P_{j})^{t} \ge
  \int_{\G_{k,d}}\int_{\G_{k,d}} \tr(P,Q)^{t} \mathrm d \mu_{k,d}(P) \mathrm
  d \mu_{k,d}(Q)
\]
and checking for equality, which implies $\mathrm{wce}_{\Pol_t,\|\cdot\|_{R_{t}}}(\{(P_{j},\omega_j)\}_{j=1}^n)=0$.


\section{Numerical experiments}\label{sec:nums}
We now aim to illustrate  theoretical results of the previous sections. The projective space $\mathcal{G}_{1,d}$ can be dealt with approaches for the sphere by identifying $x$ and $-x$. The space $\mathcal{G}_{d-1,d}$ can be identified with $\mathcal{G}_{1,d}$, so that the first really new example to be considered here is $\mathcal{G}_{2,4}$. 


We computed points $\{P_{j}^{t}\}_{j=1}^{n_{t}}\subset\G_{2,4}$, for 
$t=1,\dots,14$, with worst case error
\begin{equation*}
\mathrm{wce}_{\Pol_t,\|\cdot\|_{L_2}}(\{(P_j^t,1/n_t)\}_{j=1}^{n_t}) < 10^{-7}
\end{equation*}
by a nonlinear conjugate gradient method on manifolds, cf.~\cite[Section 3.3.1]{Graf:2013zl}, see also \cite[Section 8.3]{Absil:2008qr}. Although the worst case error may not be zero exactly, we shall refer to $\{P_{j}^{t}\}_{j=1}^{n_{t}}$ in the following simply as $t$-designs. Note that $\mathcal{G}_{2,4}$ has dimension $\dim(\G_{2,4})=4$, so that the number of
cubature points must satisfy $n_{t} \gtrsim t^{4}$. Indeed we chose
\[
  n_{t} := \Big\lfloor \frac13 \dim(\Pol_{t}(\mathcal{G}_{2,4})) \Big\rfloor = \Big\lfloor  \tfrac13 (t+1)^2(1+t+\tfrac12 t^2) \Big\rfloor.
\]
We emphasize that for $t=14$ we computed $n_{14}=8.475$ projection matrices
which almost perfectly integrate $25.425$ polynomial basis functions.

\subsection{Integration}
In what follows we consider two positive definite kernels 
\[
  \begin{aligned}
    K_{1}(P,Q) & = \sqrt{(2-\tr(PQ))^{3}} + 2 \tr(PQ), \\
    K_{2}(P,Q) & =  \exp(\tr(PQ) - 2).
  \end{aligned}
\]
It can be checked by comparison to the Bessel kernel, cf.~\cite{Brandolini:2014oz}, that the reproducing kernel Hilbert space $H_{K_{1}}$ equals
the Bessel potential space $H^{\frac72}_2(\G_{2,4})$, i.e., the corresponding norms
are comparable. In contrast, the reproducing kernel Hilbert space $H_{K_{2}}$ is
contained in the Bessel potential space $H^{s}_2(\G_{2,4})$ for any $s>2$. The worst case errors can be computed by
\[
  \begin{aligned}
    \mathrm{wce}_{H_{K_{1}},\|\cdot\|_{K_{1}}}(\{(P_{j},\frac1n)\}_{j=1}^{n})^{2} & = \frac{1}{n^{2}} \sum_{i,j=1}^{n} K_{1}(P_{i},P_{j}) - \big(2 + \frac{74}{75} \sqrt 2 - \frac{2}{5} \log(1+\sqrt 2)\big), \\
    \mathrm{wce}_{H_{K_{2}},\|\cdot\|_{K_{2}}}(\{(P_{j},\frac1n)\}_{j=1}^{n})^{2} & = \frac{1}{n^{2}} \sum_{i,j=1}^{n} K_{2}(P_{i},P_{j}) - \exp(-1) \mathrm{Shi}(1),
  \end{aligned}
\]
where $\mathrm{Shi}(x) = \int_{0}^{x} \frac{\sinh(t)}{t} \mathrm d t$ is the
hyperbolic sine integral.

In view of illustrating Proposition \ref{th:random inde}, note that random $P\in\mathcal{G}_{k,d}$ distributed according to $\mu_{k,d}$ can be derived by $P := Z(Z^\top Z)^{-1}Z^\top$, where $Z\in\R^{d\times k}$ with entries that are independently and identically standard normally distributed, cf.~\cite[Theorem 2.2.2]{Chikuse:2003aa}.

Figure~\ref{fig:random} depicts clearly the superior integration quality of the
computed cubature points over random sampling. Moreover, it can be seen that the
theoretical results in Proposition \ref{th:random inde} and Theorem \ref{th:brandolini} with \eqref{eq:general cub fomula} are in perfect accordance with
the numerical experiment, i.e., the integration errors of the random points
scatter around the expected integration error and cubature points achieve the
optimal rate of $n^{-\frac78}$ for functions in 
$H^{\frac72}_2(\G_{2,4})$.

In Figure~\ref{fig:smooth} we aim to show the contrast between the integration
of functions in $H_{K_{1}}$ and $H_{K_{2}}$ by using the computed $t$-designs.
We know by Theorem \ref{th:brandolini} that the sequence of $t$-designs with a
number of cubature points $n_{t} \asymp t^{4}$ is a QMC system for any $s>2$.
Since $H_{K_{2}}$ is contained in any Bessel potential space $H^{s}_2(\G_{2,4})$, for $s>2$, we
expect a super linear behavior in our logarithmic plots. Indeed,
Figure~\ref{fig:smooth} confirms our expectations. For $t\ge 11$, the effect of
the accuracy of the $t$-designs used becomes significant for integration of
smooth functions. For that reason, we added the dashed red line, which
represents the accuracy $10^{-7}$ of the computed $t$-designs.

\begin{figure}[t]
\begin{center}
 \includegraphics[width=9.5cm]{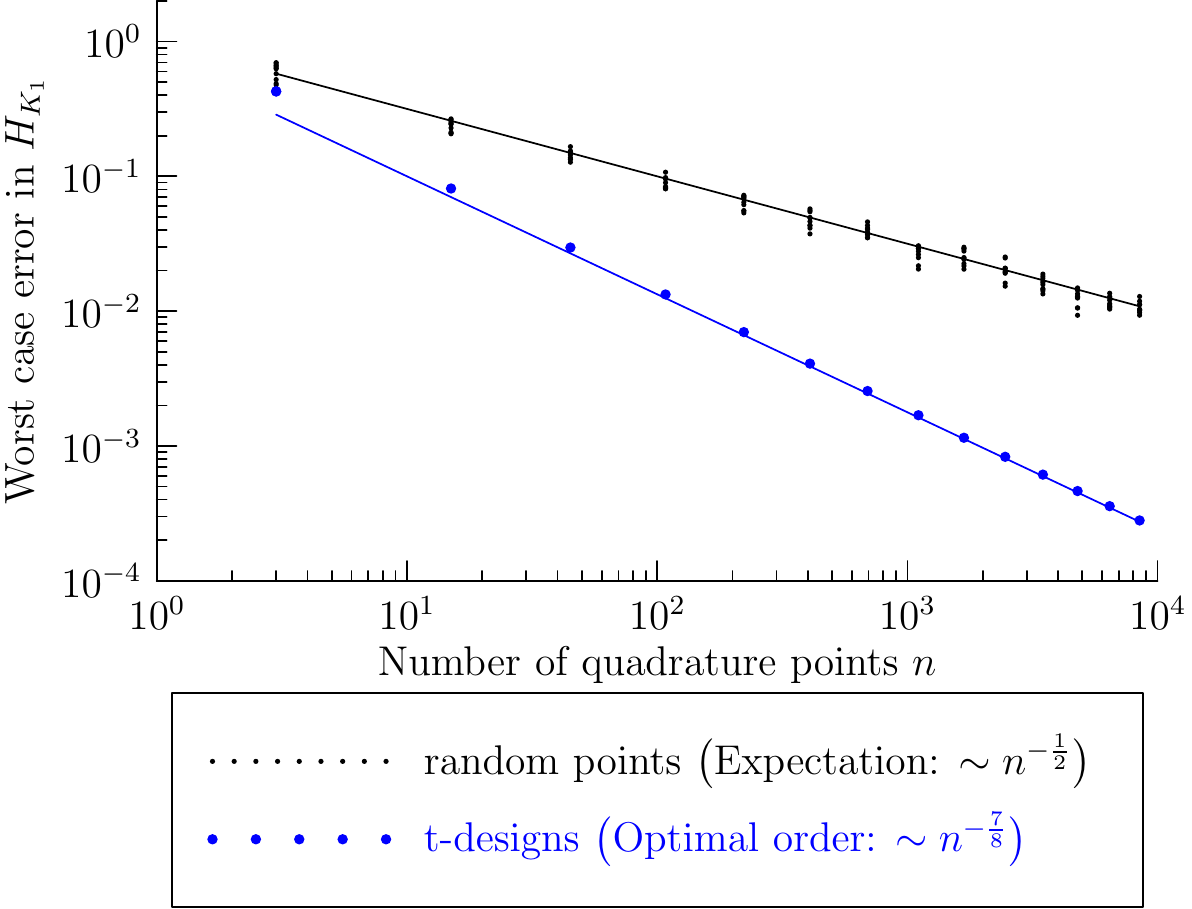}
\end{center}
\caption{ Random sampling according to $\mu_{2,4}$ vs. integration by $t$-desings.}\label{fig:random}
\end{figure}

\begin{figure}[t]
\begin{center}
 \includegraphics[width=9.5cm]{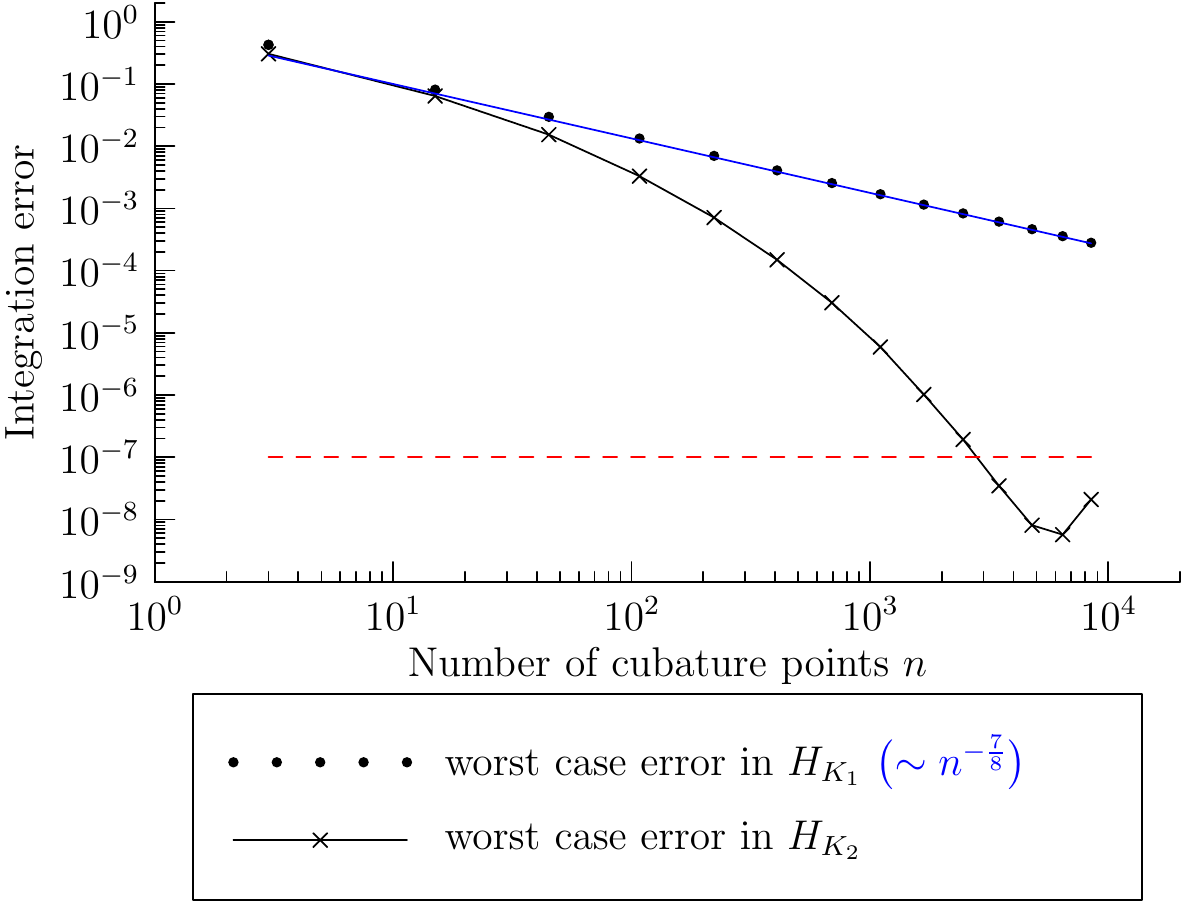}
\end{center}
 \caption{ Integration of smooth vs. integration of nonsmooth functions.}\label{fig:smooth}
\end{figure}

\subsection{Approximation}

Similar, as in the previous section we aim to approximate a smooth and a nonsmooth function, namely
\[
  f_{1}(P) = K_{1}(I_{2}, P), \qquad f_{2}(P) = K_{2}(I_{2},P), \qquad P \in
  \G_{2,4},
\]
where $K_{1}$, $K_{2}$ are from the previous section and $I_{2}$ is a projection
matrix with $2$ ones on the upper left diagonal. This time we observe that the
function $f_{1}$ is contained in $H^3_\infty(\mathcal{G}_{2,4})$ but
$f\not\in H^{3+\epsilon}_\infty(\mathcal{G}_{2,4})$, for all $\epsilon>0$. For
the smooth function $f_{2}$, we have $f_{2} \in H^{s}_\infty(\G_{2,4})$, for any
$s>0$.

Since the computed $t$-designs are with respect to $\Pol_{t}(\G_{2,4})$ and not
$\Pi_{t}(\G_{2,4})$, we need an additional scaling of $s(t)$ in $\sigma_{t}$. 
According to Example \ref{ex:first one}, the choice 
\begin{equation*}
s(t) = \sqrt{2(t^{2}+3t + 3+(-1)^{t})} -\epsilon\asymp t \sqrt 2,
\end{equation*}
for small $\epsilon >0$, 
yields $\Pi_{s(t)} \cdot \Pi_{s(t)}\subset \Pol_{2t}(\mathcal{G}_{2,4})$. For numerical experiments, we take $\epsilon$ to be smaller than the machine precision, so that it is effectively zero. Hence, in accordance with Theorem \ref{th:m 0}, we use the
following kernel based approximation
\[
  \sigma_{s(t)}(f,X_{2t}) =\frac{1}{n_{2t}}\sum_{j=1}^{n_{2t}} f(P_{j}) \sum_{l(\pi)\leq 2}
  h(s(t)^{-2}\lambda(\pi)) K_{\pi}(P_{j},\cdot ),
\] 
where  $X_{2t}=\{(P^{2t}_j,1/n_{2t})\}_{j=1}^{n_{2t}}$ and 
\begin{equation*}
h(x) = \begin{cases}
\big(1 + \exp(\frac{3 - 4 x}{2 - 6 x + 4 x^2})\big)^{-1},& 1/2<x<1,\\
1 ,&x\leq 1/2,\\
0, &\text{otherwise.}
\end{cases}
\end{equation*}

The approximation error is determined by randomly sampling altogether $50 000$ points. The first 25000 are pseudo random according to $\mu_{2,4}$. 
Since $f_{1}$ has a nonsmooth point at $I_{2}$ the maximal error is expected
around this point. Therefore, we sampled the other $25000$ from normally distributed
points around that point $I_{2}$ with variance $0.15$ and $0.5$ in the
matrix entries, i.e., we choose $Z\in\R^{4\times 4}$ with independent and identically distributed entries according to a normal distribution with mean zero and variance $0.15$ and $0.5$, respectively, and then project $I_2+Z$ onto $\mathcal{G}_{2,4}$, which we accomplished by a QR-decomposition in Matlab. 

In Figure~\ref{fig:approx}, we can observe the predicted decay in Theorem
\ref{th:m 0} for the function $f_{1} \in H^{3}_\infty(\G_{2,4})$.
Furthermore, as expected for the smooth function $f_{2}$, the error appears to
decrease super linearly.

\begin{figure}[t]
\begin{center}
 \includegraphics[width=9.5cm]{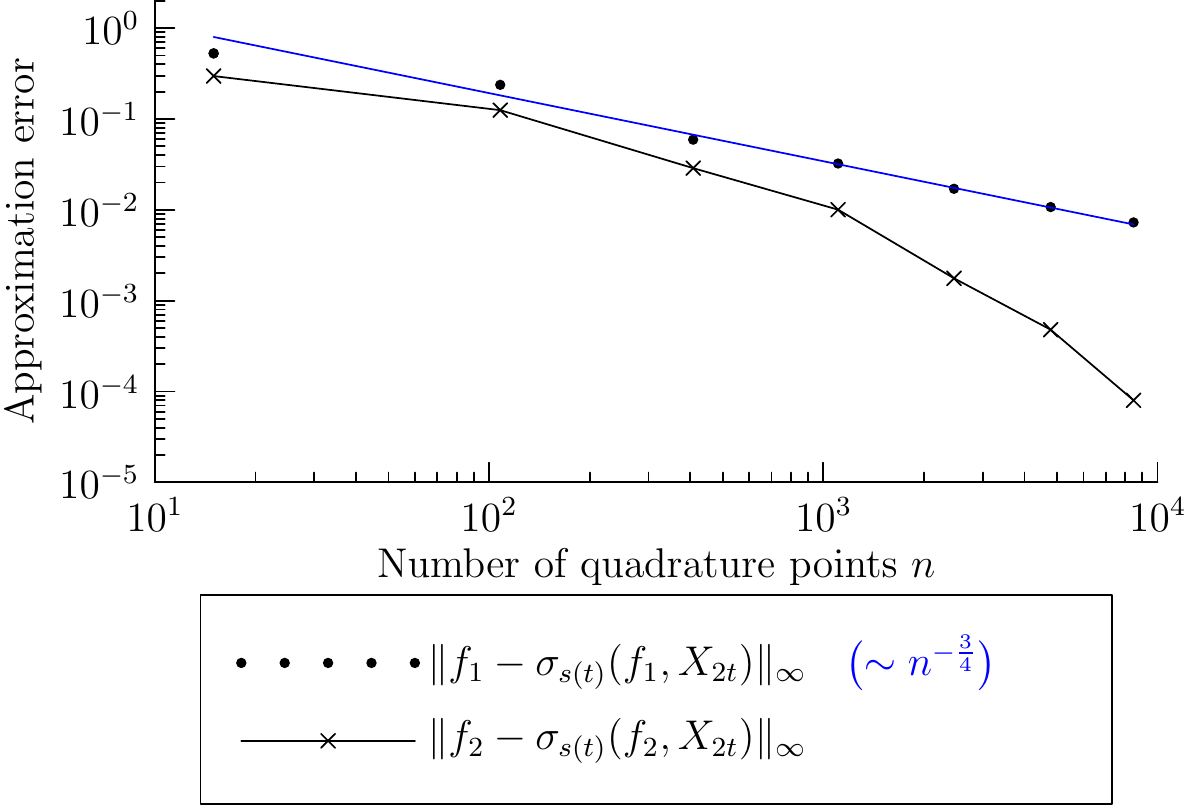}
\end{center}
\caption{ Approximation of a smooth vs. approximation of a nonsmooth
  function.}\label{fig:approx}
\end{figure}

\begin{acknowledgement}
The authors have been funded by the Vienna Science and Technology Fund (WWTF) through project VRG12-009.
\end{acknowledgement}

\bibliographystyle{amsplain}
\bibliography{../biblio_ehler2}
\end{document}